\documentclass{amsart}

\usepackage{amssymb}
\usepackage{graphicx}
\usepackage{asymptote}
\usepackage{psfrag}

\newcommand{\p}{\partial}

\newcommand{\ti}{\tilde}

\newcommand{\ra}{\rightarrow}
\newcommand{\R}{\mathbb{R}}

\renewcommand{\phi}{\varphi}
\newcommand{\lan}{\left\langle}
\newcommand{\ran}{\right\rangle}

\newtheorem{theorem}{Theorem}
\newtheorem{proposition}{Proposition}
\newtheorem{lemma}{Lemma}
\newtheorem{definition}{Definition}

\newtheorem{thm}{Theorem}[section]

\newtheorem{prop}[thm]{Proposition}
\newtheorem*{prop*}{Proposition}

\numberwithin{equation}{section}

\theoremstyle{definition}

\title[Scattering rigidity for analytic manifolds with a magnetic field]{Scattering rigidity for analytic Riemannian manifolds with a possible magnetic field}

\author[P. Herreros]{Pilar Herreros}
\author[J. Vargo]{James Vargo}
\address{Mathematisches Institut, University of M\"unster, 48149 M\"unster, Germany}
\address{Department of Mathematics, University of Washington, Seattle, WA 98195}

\begin{document}
\begin{abstract}
Consider a compact manifold $M$ with boundary $\p M$ endowed with a Riemannian metric $g$ and a magnetic
field $\Omega$.  Given a point and direction of entry at the boundary, the scattering relation $\Sigma$ determines
the point and direction of exit of a particle of unit charge, mass, and energy.  In this paper we show
that a magnetic system $(M, \p M, g, \Omega)$ that is known to be real-analytic and that
satisfies some mild restrictions on conjugate points is uniquely determined up to a natural
equivalence by $\Sigma$.  In the case that the magnetic field $\Omega$ is taken to be zero, this gives a new rigidity result
in Riemannian geometry that is more general than related results in the literature.
\end{abstract}

\maketitle

\section{An introduction including the results proved}

\subsection{Magnetic systems and scattering data}
Let $M$ be a compact, $n$-dimensional, Riemannian manifold with
boundary $\p M$ and metric $g$.  Let $\pi:T^*M\rightarrow M$ be the
natural projection of the cotangent bundle to $M$ given by
$(x,\xi)\mapsto x$.  The Hamiltonian flow of the function
$H(x,\xi)=\frac{1}{2}|\xi|_g^2$ with respect to the standard
symplectic structure of $T^*M$ is the geodesic flow.  Given a closed
$2$-form $\Omega$ on $M$, there is a corresponding {\em magnetic
geodesic flow} which is determined by the same Hamiltonian function,
but through the modified symplectic $2$-form
\begin{align*}
\omega= dx\wedge d\xi+\pi^*\Omega.
\end{align*}
The integral curves of this flow, projected to $M$ by $\pi$, are the
{\em magnetic geodesics}. The level sets of $H$ are preserved along
the flow and that implies that the magnetic geodesics have constant speed.
In this paper, we shall restrict our attention to the energy level
$H=1/2$, the unit co-sphere bundle $S^*M$. The isomorphism
$T^*M\rightarrow TM$ induced by the metric $g$ transforms the
magnetic flow to a flow on $SM$, the sphere bundle in $TM$.  Let
$\mathcal{X}$ and $\mathcal{X}_\mu$ denote the vector fields on $SM$
generating the standard geodesic flow and magnetic geodesic flow,
respectively. They are related by the equation
\[\mathcal{X}_\mu(x,v)=\mathcal{X}(x,v)+Y(x)v,\]
where $Y$ is an antisymmetric $(1,1)$ tensor on $M$ satisfying
\begin{align*}
\langle Y(x)v,w\rangle_g=\Omega(v,w).
\end{align*}
Let $\Phi^t:SM\rightarrow SM$ denote the magnetic geodesic flow.
Given a pair $(x,\theta)\in SM$, let $\gamma_{x,\theta}$ or
$\gamma_{x,\theta}(t)$ denote the magnetic geodesic with initial
point and vector $(x,\theta)$.  We also use the magnetic exponential map
\[ \exp^\mu:TM\ra M,\]
defined by $\exp^\mu_x(t\theta)=\gamma_{x,\theta}(t)$ for unit length $\theta$.
This map is $C^1$ everywhere and smooth away from $\theta = 0$.

\begin{definition}
A $4$-tuple $(M, \p M, g, \Omega)$ is called a {\bf magnetic system}.  Two such
systems $(M_i, \p M_i, g_i,\Omega_i),\, i=1,2$ with the same boundary $\p
M_1=\p M_2$ are said to be {\bf equivalent} if there exists a
diffeomorphism $\phi:M_1\rightarrow M_2$ such that:
\begin{align*}
\phi^*g_2=g_1, \qquad \phi^*\Omega_2=\Omega_1,\qquad \phi|_{\p
M}=id.
\end{align*}
The diffeomorphism $\phi$ is called a {\bf magnetic equivalence}.
\end{definition}

In this paper, the term analytic shall mean real-analytic. We define
a magnetic system $(M, \p M, g,\Omega)$ to be {\em
analytic} if there exists a magnetic system $(M', \p M', g',\Omega')$ with
$M'$ a real-analytic open manifold containing $M$ such that:
\begin{itemize}
\item[(i)]  $g'|_M=g$ and $\Omega'|_M=\Omega$,
\item[(ii)]  $g'$ and $\Omega'$ are real-analytic on $M'$,
\item[(iii)]  $\p M$ is a real-analytic submanifold of $M'$.
\end{itemize}
The extensions $(g',\Omega')$ are uniquely determined from
$(g,\Omega)$ by analytic continuation and shall henceforth be
denoted without the prime.  That is, $(g,\Omega)$ shall denote its
own extension to $M'$.

The {\em scattering magnetic rigidity question} is whether the equivalence
class of a magnetic system is determined by the scattering of its
magnetic geodesics.  Given a point of entry $x\in \p M$ and an
initial direction $\theta\in S_x M$, the scattering data tells all
future points of contact $\gamma_{x,\theta}\cap \p M$ as well as the
directions of motion at those points.  More precisely:

\begin{definition}
A pair of vectors $(v,w)\in \p SM\times \p SM$ belongs to the {\bf
scattering relation} $\Sigma$ if for some $t\geq 0$, $w=\Phi^t(v)$.
\end{definition}

After properly defining all the terms and introducing some minor
restrictions on conjugate points, it is shown that two analytic magnetic systems are
equivalent if they have the same scattering data.

This problem descends from the {\em boundary rigidity problem}
in which one seeks to recover a metric $g$ up to isometry from the
boundary distance function $\rho_g:\p M\times \p M\rightarrow
\mathbb{R}$, defined by:
\begin{align*}
\rho_g(x,y)=\inf \{ \text{length}(\gamma):\, \gamma\subset M\,
\text{is a curve joining}\, x\, \text{to}\, y\}.
\end{align*}
A particular Riemannian manifold with boundary $(M, \p M, g)$ is called
{\em boundary rigid} if all other such manifolds with the same
boundary distance function are in the same isometry class.  The
known examples of manifolds that are not boundary rigid either have
trapped geodesics (geodesics that never leave the manifold) or
conjugate points.  The most general class of manifolds conjectured
to be boundary rigid are those with the SGM property.  This property is detectable from the boundary and is equivalent to the two conditions that there are no trapped geodesics and that all geodesic segments are strong length minimizers \cite{C1}.
SGM manifolds homeomorphic to a ball with a strictly convex boundary are called
{\em simple}.  It has been conjectured that all simple manifolds are boundary rigid (e.g. \cite{Mi}), but so far, such a general statement has only been confirmed in dimension $2$ in work by Pestov and Uhlmann \cite{PU}.  In higher dimensions, Gromov showed that subregions of Euclidean space are boundary rigid \cite{Gr}.  Other spaces known to be boundary rigid are subregions of boundary rigid spaces, subregions of a hemisphere \cite{Mi}, simple symmetric spaces of negative curvature \cite{BCG}, and SGM subregions of spaces with a parallel vector field \cite{CK}.  Recently, it was shown by Burago and Ivanov that metrics sufficiently close in the $C^\infty$ topology to the Euclidean metric on a convex domain are boundary rigid \cite{BI}.  The last result implies, for example, that for any point on any Riemannian manifold, a sufficiently small neighborhood of that point is boundary rigid.

In applications, this problem is called travel time tomography and
goes back at least as far as the early 20th century, when it was proposed
to measure seismic waves to reconstruct an image of the inner structure of the earth. Besides being used
widely by geophysicists, travel time tomography is used by oceanographers and atmospheric scientists, and has potential applications in medical imaging. In
some applications, it is natural to assume that the metric is
isotropic; that is, conformal to the Euclidean.  For metrics in a
given conformal class, uniqueness was first shown for simple metrics by
Mukhometov \cite{Mu} and later by Croke for SGM metrics \cite{C1}.

The SGM and simplicity conditions are restrictive geometric
conditions which are not typically satisfied in applications.  In
more general circumstances, one can instead consider the lens data,
which consists of the scattering data together with the travel times
of the geodesics.  It was shown by Michel that the lens data is
equivalent to the boundary distance function for simple manifolds \cite{Mi}. If
trapped geodesics are allowed, then counterexamples to lens rigidity
exist \cite{CK}.  Croke showed that if a manifold is lens rigid,
then so is a finite quotient of that manifold \cite{C2}.  In \cite{V} lens
rigidity was shown if the metrics are a priori assumed to be real
analytic.  Stefanov and Uhlmann proved lens rigidity for metrics
that are a priori known to be close to a given base metric taken
from a set of generic metrics including the real analytic ones \cite{SU2}.

The magnetic rigidity problem is related to the question of whether there are more general families of curves whose defining parameters can be recovered from boundary measurements.
Magnetic geodesics and the magnetic flow were first considered by V.I. Arnold
\cite{A61} and D.V. Anosov and Y.G. Sinai \cite{AS}. Since then it has been studied from several approaches (e.g. \cite{CMP, Gi96, Gr99,NT, PP97}) and their boundary rigidity, in particular, has been studied in \cite{DPSU, He}.

One motivation for magnetic boundary rigidity is that it has potential applications to problems of geometry. In the case of surfaces, magnetic geodesics are related to the study of curves of constant geodesic curvature. This relation has been used to study the existence of closed curves with prescribed geodesic curvature (see e.g \cite{A88, Le, Sch}). Magnetic boundary rigidity was used in \cite{He} to show that a surface of constant curvature cannot be modified in a small region while keeping all the curves of a fixed constant geodesic curvature closed.

Let $\iota:\p M\hookrightarrow M$ be inclusion. Since a magnetic
equivalence fixes the points of $\p M$, the pullbacks $\iota^*g$ and
$\iota^*\Omega$ are preserved.  However the scattering relation
$\Sigma$ is not preserved because its definition depends on the set
$\p SM$ which is defined through $g|_{\p M}$. For this reason we
reparametrize $\Sigma$ through the natural diffeomorphism
\begin{align*}
\Lambda:\p SM\rightarrow \ti{\p SM}=\{(w,r)\in T\p M\times
\mathbb{R}:\, |w|_g^2+r^2=1\}.
\end{align*}
$\Lambda$ is given by $v\mapsto (v^T,\langle v,\nu\rangle_g)$, where
$v^T$ denotes the orthogonal projection of $v\in \p SM$ to $T\p M$,
and $\nu$ denotes the inward-pointing unit normal vector. The
reparametrized scattering relation, $\ti{\Sigma}\subset \ti{\p
SM}\times\ti{\p SM}$, is defined to be $\ti{\Sigma}=(\Lambda\times
\Lambda) (\Sigma).$  Given an equivalence class of magnetic systems
$\{(M, \p M, g,\Omega)\}$, we shall call $(\p
M,\,\iota^*g,\,\iota^*\Omega,\,\ti{\Sigma})$ the associated {\em
scattering data}.

An equivalent way of parametrizing the scattering data is given
through {\em boundary normal coordinates}.  Given $x_0\in \p M$,
there exists a neighborhood $N\subset \p M$ of $x_0$ and a number
$\epsilon>0$ such that the map
\begin{align*}
\exp_{\nu}:N\times (-\epsilon,\epsilon)\rightarrow M'
\end{align*}
given by $\exp_{\nu}(x',x^n)=\exp_{x'}(x^n\nu)$ is an injective
immersion. Here $\exp$ denotes the exponential map corresponding to
the Riemannian metric $g$, not the magnetic exponential $\exp^\mu$.  Moreover, if $x^\alpha:1\leq \alpha\leq
n-1$ are local coordinates for $\p M$, then the metric tensor has
the form
\begin{align*}
g=g_{\alpha\beta}(x)dx^{\alpha}dx^{\beta}+(dx^n)^2.
\end{align*}
$(x',x^n)$, as local coordinates for $M$ near $x_0$ are called {\em
boundary normal coordinates}, or {\em semigeodesic coordinates}.
Note that the curves given by $x'=const$ are geodesics (not magnetic
geodesics) of $(M,g)$ normal to $\p M$.

Since $\p M$ is compact, it is possible to extend the boundary
normal coordinates to a collared neighborhood of the entire
boundary. Indeed, if $\epsilon$ is sufficiently small, then we
obtain a diffeomorphism
\begin{align}\label{BNC}
\exp_{\nu}:\p M\times (-\epsilon,\epsilon)\rightarrow V,
\end{align}
where $V\subset M'$ consists of the points whose distance to $\p M$
is less than $\epsilon$.  For a detailed proof, see \cite{V}.  It
shall be convenient to set
\begin{align*}
M'=M\cup V.
\end{align*}
Finally, when discussing two magnetic systems with the same
boundary, we shall assume that the value of $\epsilon$ is chosen
sufficiently small so that \eqref{BNC} is a valid diffeomorphism for
both.

Note that the mapping $\Lambda$ is equal to the pullback of vectors
at the boundary by $\exp_{\nu}$.
\begin{align*}
\Lambda=\exp_{\nu}^*|_{\p SM}.
\end{align*}

\begin{definition}
$(M,\p M, g,\Omega)$ is called {\bf non-trapping} if for all
$(x,\theta)\in SM$, there exist $T_-<0$ and $T_+>0$ such that
$\gamma_{x,\theta}(T_\pm)\in M'\setminus M$; that is, every magnetic
geodesic must eventually leave $M$ in both directions.
\end{definition}

In this paper, we shall assume that all magnetic systems are
non-trapping.  Let $v=(x,\theta)\in \p SM$, and let $\ell(v)=T$ be
defined by
\begin{align*}
T=\inf\, \{t\geq 0: \gamma_v(t)\notin M\}.
\end{align*}
Then $\gamma_v(T)$ and $\gamma_v'(T)$ are respectively called the
terminal point and direction of $\gamma_v$ in $M$.  The mapping
\begin{align*}
\ell:\p SM\rightarrow \mathbb{R}
\end{align*}
is called the {\em travel time map of M}.  The corresponding map
$\ti{\ell}:\ti{\p SM}\rightarrow \mathbb{R}$ given by
$\ti{\ell}=\ell\circ \Lambda^{-1}$ is called the {\em travel time
data} of the equivalence class of $(M, \p M, g,\Omega)$.
\begin{definition}
We say that $v$ satisfies {\bf condition A} if there are no points
in $\p M\cap \{\gamma_{v}(t):0\leq t\leq T\}$ that are conjugate to
$x$ along $\gamma_{v}$.  In this paper, conjugate shall always mean with respect to the magnetic flow.  
We shall say that an analytic magnetic
system $(M, \p M, g,\Omega)$ satisfies {\bf condition $\hat{A}$ } if there
exists at least one vector in each connected component of $\p SM$
that satisfies condition A.
\end{definition}

Let $\mathcal{G}$ be the collection of magnetic geodesic segments in
$M$ whose endpoints lie in $\partial M$ and do not intersect $\partial M$ at any other point. Let $\gamma\in \mathcal{G}$ be
an element with initial vector $v\in SM$, travel time $T$, and
final point $y=\gamma(T)$.
\begin{definition}
We say that $\gamma$ satisfies {\bf condition B} if there are no
points on $\gamma$ that are conjugate of order $n-1$ to $y$.  We
shall say that an analytic magnetic system $(M, \p M, g,\Omega)$ satisfies
{\bf condition $\hat{B}$} if the set of magnetic geodesic segments in
$\mathcal{G}$ satisfying condition B is dense in $\mathcal{G}$.
\end{definition}
Recall that if $x$ and $y$ belong to $M$ and $y=\exp^\mu_x(v)$, then $x$
and $y$ are conjugate if $d_v \exp^\mu_x(v)$ is singular.  The order, or multiplicity, of
conjugacy is equal to the dimension of the kernel with $n-1$ being
the largest possible. For example, two antipodal points on $S^n$ (with $\Omega=0$) are
conjugate of order $n-1$.

\begin{lemma}\label{openA}
Condition A is open in $\p SM$.  Condition B is open in
$\mathcal{G}$.
\end{lemma}
\begin{proof}
The openness of condition B is clear.  Let $v_0\in \p SM$ satisfy
condition A. Let $K=\gamma_{v_0}\cap \p M$. By compactness, for any
open neighborhood $N$ of $K$, there exists an open $\ti{U}$ about
$v_0$ in $S M'$ such that for all $v\in \ti{U}$, $\gamma_v\cap \p
M\subset N$.

Since $x_0$, the basepoint of $v_0$, is not conjugate to any point
of $K$, we may choose $N$ and the corresponding neighborhood
$\ti{U}$ sufficiently small so that for each $v\in \ti{U}$, its
basepoint $x$ will not be conjugate to any point of $\gamma_v\cap
N$.  But since $\gamma_v\cap \p M\subset N$, we find that $x$ will
not be conjugate to any point of $\gamma_v\cap \p M$.
\end{proof}

\begin{theorem}\label{maintheorem}
Let $(M_i, \p M, g_i,\Omega_i),\, i=1,2$ be two non-trapping, analytic
magnetic systems satisfying condition $\hat B$ and assume that
$(M_1, \p M, g_1,\Omega_1)$ satisfies condition $\hat A$.  Then if the two
systems have the same scattering data, they must be equivalent.
\end{theorem}

The first step of the proof is to show that the mapping
$\phi_0:V_1\rightarrow V_2$ given by
\begin{align}\label{phi0}
\phi_0=\exp_{\nu_2}\circ\exp_{\nu_1}^{-1}
\end{align}
is a magnetic equivalence fixing the points of $\p M$.  To do this,
we must use condition A to show that the coefficients of
$(g,\Omega)$ are uniquely determined by the scattering data in
boundary normal coordinates.  The main step is to prove the
following theorem.

\begin{theorem}\label{boundary determination}
Let $(M_i, \p M, g_i,\Omega_i),\, i=1,2$ be non-trapping, analytic magnetic
systems with the same scattering data, and assume that
$v_0=(x_0,\theta_0)\in S \p M$ satisfies condition A in $M_1$. Then
at the point $x_0$, the coefficients of $(g_1,\Omega_1)$ and
$(g_2,\Omega_2)$, expressed in boundary normal coordinates, have the
same jets at $x_0$.  That is, their values at $x_0$ are the same,
and the values of all their derivatives of all orders at $x_0$ are
the same.
\end{theorem}

 This theorem is a direct generalization of a result by Plamen
Stefanov and Gunther Uhlmann for Riemannian metrics (with no
magnetic fields).  We extend their result by allowing magnetic
fields and by only assuming the scattering data and no magnetic
analog of length data.  In order to generalize the proof, some extra
condition is needed for the case when the magnetic geodesics near
$\gamma_{v_0}$ are short. For this paper the convenient extra
assumption is to assume that the system be analytic.  However it
should be noted that this is not at all an essential condition and
could be easily be replaced by a certain weak convexity condition.
See the remark after the proof for details.

A similar theorem was proved for magnetic systems by \cite{DPSU} in
which they assume strict convexity of the boundary with respect to
magnetic geodesics and take a certain action function as the given
data. We use a similar action function to prove this theorem, but do
not otherwise use the same techniques since we allow for the
possibility that $\gamma_{v_0}$ be a long magnetic geodesic. Rather
we follow Stefanov and Uhlmann's proof, the idea of which is to
consider a function $\rho(x,y)$ which gives travel times between
pairs of points. Fixing $y$, the function satisfies the eikonal
equation:
\begin{align*}
g^{ij}(x)(\p_{x_i}\rho)(\p _{x_j}\rho)=1.
\end{align*}
Using this equation, the scattering data tells us just enough about
the derivatives of $\rho$ to reconstruct the jet of $g$. The proof
here follows the same program, but with modifications to account for
magnetic fields and the lack of explicit travel time data. We
construct an analogous action function $\rho(x,y)$ which satisfies
an eikonal-type equation (equation \eqref{eikonal}) whose
coefficients come from $g$ and $\Omega$. We then use the scattering
data to infer enough information about the derivatives of $\rho$ to
recover the jets of $g$ and $\Omega$.

The next step in the proof of Theorem \ref{maintheorem} is the proof
of the following theorem.
\begin{theorem}\label{length recovery}
Let $(M_i,\p M, g_i,\Omega_i),\, i=1,2$ be two non-trapping, analytic
magnetic systems that satisfy condition B and have the same
scattering data. If $\phi_0: V_1\rightarrow V_2$ is an equivalence,
then corresponding elements of $\mathcal{G}_1$ and $\mathcal{G}_2$
have the same travel times.  In particular, the two systems have the
same travel time data $\ti{\ell}$.
\end{theorem}
Elements of $\mathcal{G}_1$ and $\mathcal{G}_2$ are corresponding if
their initial and final vectors are related by $\phi_0$.  The next
theorem finishes the proof of \ref{maintheorem}.

\begin{theorem}\label{equivalence construction}
Let $(M_i,\p M, g_i,\Omega_i),\, i=1,2$ be two non-trapping, analytic
magnetic systems with the same scattering and travel time data.  If
$\phi_0: V_1\rightarrow V_2$ is an equivalence, then $\phi_0$
extends to an equivalence $\phi:M_1'\rightarrow M_2'$.
\end{theorem}

\section{Theorem \ref{boundary determination}: Recovery of the magnetic system in a neighborhood of the boundary}
\subsection{Defining an action function on magnetic geodesic
segments}The first step shall be to construct an action function
which serves the role of the travel time function used by Stefanov
and Uhlmann. Subscripts shall be omitted as long as all statements
apply equally to both manifolds. We consider our magnetic system to
be extended to an open manifold $M'$ containing $M$.  Let
$\gamma=\{\gamma(t):0\leq t\leq T\}$ be a magnetic geodesic segment
in $M'$.  For our proof, we shall need a $1$-form $\zeta$, defined
in a neighborhood of $\gamma$, satisfying $d\zeta=\Omega.$  By the
Poincar\'e lemma, such a $1$-form exists if $\gamma$ has no
self-intersections.  If $\gamma$ does intersect itself, then we
circumvent the topological problem through the construction
described in the following paragraphs.

Let $\ti{\gamma}$ denote the segment in $\mathbb{R}^n$ given by
\begin{align*}
\ti{\gamma}(t)=(0,\dots,\, 0,\,t)\quad \text{for} \quad 0\leq t\leq
T.
\end{align*}
Then there exists a neighborhood $U\subset
\mathbb{R}^n$ of $\ti{\gamma}$ and an immersion $\psi:U\rightarrow
M'$ that satisfies $\psi(0,t)=\psi\circ\ti{\gamma}(t)=\gamma(t)$.
For example, we can construct such a map in the following way.
Extend $\gamma$ to an open interval containing $[0,T]$ and choose
vectors $X_i(t)$ along $\gamma(t)$ such that
$(X_1(t),\dots,X_{n-1}(t),\gamma'(t))$ are linearly independent
for all $t$.  Then define
\begin{align*}
\psi(x^1,\dots,x^{n-1},x^n)=\exp_{\gamma(x^n)} \left( \sum\limits_{i=1}^{n-1}x^iX_i(x^n) \right ).
\end{align*}
The topology of $U$ is trivial so by the Poincar\'e lemma, there
exists $\zeta$ such that $d\zeta=\psi^*\Omega$.  We call $\zeta$
a {\em magnetic potential} in a neighborhood of $\gamma$.

Let $\ti{x}$ be coordinates for $U$, and $(\ti{x},\ti{\xi})$ the
corresponding natural coordinates for $T^*U$.  Using $\psi$ we can
pull back structures from $M'$ to $U$.  In particular, we obtain an
exact symplectic form
\begin{align*}
\ti{\omega}=d\ti{x}\wedge
d\ti{\xi}+\psi^*\Omega=-d\,(\ti{\xi}\,d\ti{x}-\zeta)
\end{align*}
and a Hamiltonian function
$\ti{H}=\frac{1}{2}|\ti{\xi}|^2_{\ti{g}}$. Together, these generate
Hamiltonian curves whose projections to $U$ are the magnetic
geodesics of the magnetic system $(U,\psi^*g,\psi^*\Omega)$.

We parametrize the magnetic geodesic segments in $M'$ near $\gamma$
by their initial vector $v\in SM'$ and their length $\tau$.
\begin{align*}
(v,\tau)\mapsto \gamma_{v,\tau}= \{\gamma_v(t): 0\leq t\leq \tau\}.
\end{align*}
Let $(v_0,\tau_0)$ correspond to the original segment $\gamma$ about
which $\zeta$ was constructed.  For $(v,\tau)$ in some neighborhood
$\mathcal{M}$ about $(v_0,\tau_0)$, the magnetic geodesic segments
$\gamma_{v,\tau}$ in $M'$ can be uniquely pulled back via $\psi$ to
magnetic geodesic segments $\ti{\gamma}_{v,\tau}$ in $U$. Let
$\ti{c}_{v,\tau}$ denote the corresponding Hamiltonian curve in
$T^*U$.

Letting $\mathcal{N}$ denote the neighborhood of magnetic geodesic
segments
\begin{align*}
\mathcal{N}=\{\gamma_{v,\tau}:(v,\tau)\in\mathcal{M}\},
\end{align*}
we define the action functional
$\bf{A}_{\zeta}:\mathcal{N}\rightarrow \mathbb{R}$ by
\begin{align*}
\bf{A}_\zeta[\gamma_{v,\tau}]=\tau-\int_{\ti{\gamma}_{v,\tau}}\zeta.
\end{align*}
In the case that $\Omega$ is exact, $\zeta$ can be defined on all of
$M'$ so that $\bf{A}_\zeta$ is well-defined for all magnetic
geodesics.  But in general this is not possible so the action can
only be defined near one fixed curve.  In either case,
$\bf{A}_\zeta$ depends on the choice of $\zeta$ which is not
uniquely determined by the magnetic system.

\begin{proposition}
\begin{align*}
\bf{A}_\zeta[\gamma_{v,\tau}]=\int_{\ti{c}_{v,\tau}}(\ti{\xi}\,d\ti{x}-\pi^*\zeta)
\end{align*}
\end{proposition}
\begin{proof}
Along the curve $\ti{c}_{v,\tau}$, we have
$\ti{\xi}\,d\ti{x}=|\ti{\xi}|^2_g\,dt=dt$.  Therefore, noting that
$\ti{\gamma}_v(t)=\pi\circ\ti{c}_v(t)$, the integral on the right is
equal to
\begin{align*}
\tau-\int_{\ti{c}_{v,\tau}}\pi^*\zeta=\tau-\int_{\ti{\gamma}_{v,\tau}}\zeta.
\end{align*}
\end{proof}

\subsection{The first variation of $\bf{A}_{\zeta}$}
Let $\gamma_s(t): -\epsilon\leq s\leq \epsilon,\,\,0\leq t \leq
\tau_s$ be a smooth $1$-parameter family of unit-speed magnetic
geodesics in $\mathcal{N}$. Correspondingly, there is a smooth
$1$-parameter family of curves $\ti{c}_s(t)$ in $T^*U$, and by the
lemma, we have
\begin{align*}
\phi(s)=\bf{A}_\zeta[\gamma_s]=\int_{\ti{c}_s}(\ti{\xi}\,d\ti{x}-\pi^*\zeta).
\end{align*}
Let $a(s)$ be the curve $\ti{c_s}(0)$ and $b(s)$ be the curve
$\ti{c_s}(\tau_s)$.  Now let $0<|h|<\epsilon$, and consider the
surface $\sigma$ with parametrization $(s,t)\mapsto\ti{c}_s(t):
0\leq s\leq h,\,\,0\leq t\leq \tau_s$.

According to Stokes' theorem,
\begin{align*}
\left(\int_{\ti{c}_0}+\int_b-\int_{\ti{c}_h}-\int_a\right)(\ti{\xi}\,d\ti{x}-\pi^*\zeta)=\int_\sigma
(d\ti{\xi}\wedge d\ti{x}-\pi^*\Omega).
\end{align*}
It follows that
\begin{align*}
\frac{\phi(h)-\phi(0)}{h}=\frac{1}{h}\left(\int_b-\int_a\right)(\ti{\xi}\,d\ti{x}-\pi^*\zeta)-\frac{1}{h}\int_\sigma
(d\ti{\xi}\wedge d\ti{x}-\pi^*\Omega).
\end{align*}
The surface integral on the right-hand side is equal to $0$. To show
this we note that the $2$-form being integrated is $-\,\ti{\omega}$,
the symplectic form. We also note that the Hamiltonian vector field
$X_{\ti{H}}=\frac{\p}{\p t}c_s(t)$ lies tangent to $\sigma$.  By
definition,
\begin{align*}
\ti{\omega}(X_{\ti{H}},\,\cdot\,)=d\ti{H}.
\end{align*}
Since $\sigma$ is contained in the level surface $\ti{H}=1/2$, we
conclude that the pull back of $\ti{\omega}$ to $\sigma$ must be
$0$.  Hence we obtain
\begin{align*}
\frac{\phi(h)-\phi(0)}{h}=\frac{1}{h}\left(\int_b-\int_a\right)(\ti{\xi}\,d\ti{x}-\pi^*\zeta).
\end{align*}
Taking the limit as $h\rightarrow 0$, we find
\begin{align*}
\phi'(0)=(\ti{\xi}\,d\ti{x}-\pi^*\zeta)|\,^{b'(0)}_{a'(0)}
\end{align*}
$\ti{c}_s(t)=(\ti{\gamma}_s(t),\ti{\xi}_s(t))$ where $\ti{\xi}_s(t)$
is the covector corresponding to $\p_t\ti{\gamma}_s(t)$ by $g$.
Therefore
\begin{align}\label{1stvar}
\phi'(0)=[\langle
\p_t\ti{\gamma}_0,\,\cdot\,\rangle_g-\zeta\,]|^{\pi_*b'(0)}_{\pi_*a'(0)}
\end{align}
In terms of the family of magnetic geodesics
$\{\ti{\gamma}_s(t)\}_s$, $\,\pi_*a'(0)$ and $\pi_*b'(0)$
are just the variations in the initial and terminal point,
respectively, of the curves.

Now suppose that $\gamma=\gamma_{v_0,\tau_0}$ has endpoints $x_0$
and $y_0$ which are not conjugate to each other.  Then for $(x,y)$
in a sufficiently small neighborhood around $(x_0,y_0)$, there
exists a magnetic geodesic $\gamma_{x,y}$ joining $x$ to $y$ which
is smoothly dependent on $x$ and $y$. Let $\theta=\theta_{x,y}\in
T_x M'$ be its initial vector and define the smooth function
\begin{align*}
\rho(x,y)=\bf{A}_\zeta[\gamma_{x,y}].
\end{align*}
Let $w\in T_xM'.$  If we fix $y$ and take the differential of $\rho$
with respect to $x$, then by \eqref{1stvar}, we obtain:
\begin{align}\label{rhoder}
\langle d_x\rho,w\rangle=-\langle
\theta,w\rangle_g+\langle\zeta,w\rangle
\end{align}
Solving for $\theta$, we find
\begin{align}
-\langle \theta,\,\cdot\,\rangle_g=d_x\rho-\zeta .
\end{align}
In particular since $\gamma$ is a unit-speed curve, $\theta$ has
unit length, so
\begin{align*}
|d_x\rho-\zeta|_g^2=1.
\end{align*}
Written another way,
\begin{align}\label{eikonal}
g^{ij}(x)(\p_{x^i}\rho-\zeta_i(x))(\p_{x^j}\rho-\zeta_j(x))=1.
\end{align}
This is the eikonal-type equation that we need.

\subsection{Choosing compatible magnetic potentials} Now we consider again our two manifolds
$M_1$ and $M_2$ with their respective metrics and magnetic fields
and extensions $M_1'$ and $M_2'$.  We let $v_0=(x_0,\theta_0)\in S
\p M$ be a vector satisfying the hypothesis of the theorem. For each
$i=1,2$, let $\gamma_i\subset M_i'$ be the magnetic geodesic with initial vector $v_0$, and let $\zeta_i$ be an arbitrarily chosen
magnetic potential in a neighborhood of $\gamma_i$. The next
lemma shows that the two potentials can be made compatible with each
other in a certain sense.  Recall that $(x',x^n)$ denote boundary
normal coordinates near $x_0$.

\begin{lemma}\label{potatbo}
If the two magnetic systems have the same scattering data, then
there exist magnetic potentials $\zeta_1,\zeta_2$ defined in
neighborhoods of $\gamma_1$ and $\gamma_2$ respectively
which satisfy
\begin{align*}
\iota_1^*\zeta_1=\iota_2^*\zeta_2,\quad \text{and}
\end{align*}
\begin{align*}
\langle\zeta_1,\p_{x^n}\rangle=\langle\zeta_2,\p_{x^n}\rangle,
\end{align*}
in a neighborhood of the point $x_0$.
\end{lemma}
\begin{proof}
Starting with arbitrary $\zeta_1$ and $\zeta_2$, note that
\begin{align*}
d(\iota_1^*\zeta_1-\iota_2^*\zeta_2)=\iota_1^*\Omega_1-\iota_2^*\Omega_2=0.
\end{align*}
Therefore, there exists a smooth function $f_0$ defined in a
neighborhood of $x_0$ in $\p M$ such that
\begin{align*}
\iota_1^*\zeta_1-\iota_2^*\zeta_2=df_0.
\end{align*}
We extend $f_0$ in a neighborhood of $\p M$ according to the first
order PDE
\begin{align*}
\p_{x^n}f  &=
\langle\zeta_1,\p_{x^n}\rangle-\langle\zeta_2,\p_{x^n}\rangle, \\
f|_{\p M}&=f_0.
\end{align*}
Away from the boundary we apply a cutoff function so that $f$ is
well-defined over a neighborhood of $\gamma_1$.  Letting $\zeta_1'=\zeta_1+df$, we
find that $\zeta_1'$ and $\zeta_2$ satisfy the requisite conditions.
\end{proof}

\subsection{Proof of Theorem \ref{boundary determination}}
For the remainder of the proof, subscripts distinguishing the two
manifolds shall be omitted.  We are given a magnetic system $(M,\, \p M, \,
g,\, \Omega)$ with known scattering data $(\p M,\, \iota^*g,\,
\iota^*\Omega,\, \ti{\Sigma})$ and extension $(M',\p M, \, g,\, \Omega)$.
We will be working in boundary normal coordinates near a point
$x_0\in \p M$.  In these coordinates, we may regard $g|_{\p M}$ as
known since $\iota^*g$ is included in the scattering
data.

We start with a lemma.  Let $v_0\in \p SM$ be a unit vector at the
boundary, and let $\zeta$ be a magnetic potential in a neighborhood
of $\gamma_{v_0}$. Let $y_0$ be a point on $\gamma_{v_0}\cap \p M$
with $y_0\neq x_0$ and assume that $x_0$ and $y_0$ are not conjugate
along $\gamma_{v_0}$. Then there exist respective neighborhoods
$V,\,W$ about $x_0,\,y_0$ in the extended manifold $M'$ such that
for any $x\in V,\, y\in W$, there is a magnetic geodesic
$\gamma_{x,y}$ connecting them.  It depends smoothly on the
endpoints and is equal to $\gamma_{v_0}|_{[x_0,y_0]}$ if
$(x,y)=(x_0,y_0)$.  As discussed above, the function
$\rho(x,y)=\bf{A}_\zeta[\gamma_{xy}]$ is smooth with derivative
given in equation \eqref{rhoder}.

According to lemma \ref{potatbo}, we may regard $\iota^*\zeta$ as
known data. What's more, we may regard $\zeta(\p_{x^n})$ as known in
a neighborhood of $x_0$ even off the boundary. In particular,
$\iota^*\zeta$ and $\zeta(\p_{x^n})|_{\p M}$ together give us
$\zeta|_{\p M}$ in boundary normal coordinates. If we let
$\hat{\zeta}$ denote the vector field corresponding to $\zeta$ via
$g$, then that is also known on $\p M$.  What we need for the
theorem is the full jet of $g$ and $\zeta$ at $\p M$. But since we
already know $g$ and $\zeta$ on $\p M$, what remains is to recover
their derivatives in the normal direction. That is, we need:
\begin{align*}
\p_{x^n}^kg|_{\p M},\quad \p_{x^n}^k\zeta|_{\p M},\quad k=1,2,\dots
\end{align*}
Actually, what we shall get are these
normal derivatives on a half-neighborhood of $x_0$ in $\p M$. This
is enough to apply tangential derivatives to get the full jets at
$x_0$.

\begin{lemma}
Let $v_0\in \p SM$ be a unit vector over the point $x_0\in \p M$
which points strictly inwards, i.e. $\langle v_0,\nu\rangle_g>0$.
Let $y_0$ be the first point of intersection between $\gamma_{v_0}$
and $\p M$.  Then if $x_0$ and $y_0$ are not conjugate, the scattering
data uniquely determine $d_x\rho(x,y_0)$ for $x\in \p M$ sufficiently
close to $x_0$.
\end{lemma}
\begin{proof}
The existence of $y_0$ is guaranteed by the fact that $v_0$ is inward
pointing and that our systems are non-trapping.  Let $\tau>0$ denote
the travel time of the magnetic geodesic from $x_0$ to $y_0$. Then
$\gamma_{v_0}(\tau)=y_0$. Let $w_0=\gamma_{v_0}'(\tau)$, the vector
tangent to the curve at $y_0$.  It should be noted that both $y_0$ and
$w_0$ are determined by the scattering relation.  For small
$\epsilon>0$, let
\[C_\epsilon=\{w\in S_{y_0}M: \langle
w,w_0\rangle_g>1-\epsilon,\,\langle w,\nu\rangle_g<0\}.\]
Since $y_0$ is the first point of intersection of $\gamma_{v_0}$ with $\p M$, it
is clear that $\langle w_0,\nu\rangle_g\leq 0$ with equality if an
only if $w_0$ is tangent to $\p M$.  Therefore, $C_\epsilon$ is
either an open cone or a half-open cone about $w_0$, depending on
whether $w_0$ is tangent to $\p M$.  Note that in the case that it
is only half-open, $w_0$ is a limit point, but not actually a member
of $C_\epsilon$.

The elements of $C_\epsilon$ are strictly outward pointing, so they
are terminal vectors of magnetic geodesics.  Since $v_0$ is
transverse to $\p M$, we see that $\{(v,w)\in\Sigma:w\in
C_{\epsilon}\}$ gives a one-to-one correspondence between vectors
$v$ near $v_0$ with vectors $w\in C_\epsilon$.  Let $\kappa$ denote
the composition of the corresponding maps
\begin{align*}
C_\epsilon &\rightarrow \p SM \rightarrow \p M, \\
w &\mapsto v  \mapsto \pi(v),
\end{align*}
where $\pi(v)$ is projection to the base point.  Since $x_0$ and $y_0$
are not conjugate, $\kappa$ is a diffeomorphism. Also, it is
determined by the scattering data.  Its image, which we will denote
by $\mathcal{H}$, is either a half-neighborhood or full neighborhood
of $x_0$ depending on whether $C_\epsilon$ is a half-cone or full
cone about $w_0$.

The function $\rho(x,y_0)$ is well-defined for $x\in M'$ in an open
neighborhood about $x_0$.  We may assume that $\mathcal{H}$ is taken
small enough to lie inside this neighborhood.  Then from the
previous discussion, we see that for $x\in \mathcal{H}\subset \p M$,
there exists a unique vector $v_x\in S_xM$, close to $v_0$,
such that $y$ lies on $\gamma_{v_x}$. Indeed $v_x$ is a point in the
image of the map ${C_\epsilon}\rightarrow \p SM$. By equation
\eqref{rhoder},
\begin{align*}
\nabla_x\rho(x,y_0)=\hat\zeta_x-v_x.
\end{align*}
The terms on the right are uniquely determined, since $v_x$ is known
from the scattering data, and $\hat{\zeta}$ is known on $\p M$.
Therefore, since $g|_{\p M}$ is known, we find $d_x\rho(x,y_0)$ at
points $x\in\mathcal{H}$ by lowering the index of
$\nabla_x\rho(x,y_0)$.
\end{proof}

\begin{proof}[Proof of Theorem \ref{boundary determination}]
Let $v_0\in S\p M$ satisfy condition A. By lemma \ref{openA}, the
vector
\begin{align}\label{vee}
v_s(x_0)=\cos(s)v_0+\sin(s)\nu
\end{align}
will also satisfy condition A for suffiently small $s\geq 0$. We let
$y_s$ denote the first point of intersection between
$\gamma_{v_s(x_0)}$ and $\p M$, and we let $w_s$ be the tangent to the
magnetic geodesic at $y_s$. By compactness, there is a sequence
$s_m\rightarrow 0$ such that $(y_{s_m},w_{s_m})$ converges to a vector
$(y_*,w_*)$ tangent to $\gamma_{v_0}$ and lying over a point of $\p M$.
We split the proof of the theorem into the two possible cases that
$w_*=v_0$ or $w_*\neq v_0$ which correspond to whether the length of
the magnetic geodesic $\gamma_{v_0}$ from $x_0$ to $y_*$ has length $0$ or not.

{\bf Case I}: $w_*\neq v_0$.

Let $\rho_m(x)$ denote the function $\rho(x,y_{s_m})$.  Let
$\mathcal{H}_m$ be the corresponding (half)-neighborhood of $x_0$
referred to in the previous lemma.  For $x\in M'$ nearby $x_0$, let
$v_{s_m}(x)$ be the initial vector of the magnetic geodesic
connecting $x$ to $y_{s_m}$, and let $\hat{v}_{s_m}$ be the
corresponding covector.
\begin{align*}
v_{s_m}(x)=\hat\zeta_x-\nabla_x\rho_m(x) \\
\hat{v}_{s_m}(x)=\zeta_x-d_x\rho_m(x).
\end{align*}
For $x\in\mathcal{H}_m$, both $v_{s_m}(x)$ and $\hat{v}_{s_m}(x)$ are uniquely
determined by the scattering data.

We make the following conventions of notation. Near $x_0$ our
boundary normal coordinates shall be
\begin{align*}
\{(x^\alpha,x^n):\,\,\alpha=1,\dots,n-1\}.
\end{align*}
A subscript or superscript of $\alpha$ or $n$ shall mean the
corresponding tensor component: e.g.
$\p_n=\p_{x^n},\,\zeta_\alpha=\zeta(\p_\alpha),$ etc. By equation
\eqref{eikonal}, we have in boundary normal coordinates:
\begin{align*}
g^{\alpha\beta}(\zeta_\alpha-\p_{\alpha}\rho_m)(\zeta_{\beta}-\p_{\beta}\rho_m)
+(\theta_n-\p_n\rho_m)^2=1,
\end{align*}
or
\begin{align*}
g^{\alpha\beta}(\hat{v}_{s_m})_\alpha(\hat{v}_{s_m})_\beta+(\hat{v}_{s_m})_n^2=1.
\end{align*}
This equation is valid for all $x$ which are sufficiently close to
$x_0$ in the extended manifold $M'$. For $x\in \mathcal{H}_m$, the
coefficients of the equation are known.  We shall find the normal
derivatives of $g$ and $\zeta$ by successively applying $\p_n$ to
this equation and solving for all unknown terms at each step.

For the first step, we apply $\p_n$ and obtain the equation:
\begin{align}\label{pn}
(\p_ng^{\alpha\beta})(\zeta_\alpha-\p_{\alpha}\rho_m)(\zeta_{\beta}-\p_{\beta}\rho_m)
+2g^{\alpha\beta}(\p_n\zeta_\alpha-\p_n\p_{\alpha}\rho_m)(\zeta_{\beta}-\p_{\beta}\rho_m)
&
\\ + 2(\zeta_n-\p_n\rho_m)(\p_n\zeta_n-\p_n^2\rho_m) & =0.\notag
\end{align}
Splitting the second term and rearranging yields
\begin{align}\label{iterate1}
(\p_ng^{\alpha\beta})(\hat{v}_{s_m})_\alpha(\hat{v}_{s_m})_\beta
&+2(g^{\alpha\beta}\p_n\zeta_\alpha)(\hat{v}_{s_m})_\beta=\\
&2g^{\alpha\beta}(\hat{v}_{s_m})_\beta\p_n\p_\alpha\rho_m
-2(\hat{v}_{s_m})_n(\p_n\zeta_n-\p_n^2\rho_m).\notag
\end{align}
Next, we let $m\rightarrow \infty$.  On the left side we obtain:
\begin{align*}
(\p_ng^{\alpha\beta})(\hat{v}_0)_\alpha(\hat{v}_0)_\beta+
2(g^{\alpha\beta}\p_n\zeta_\alpha)(\hat{v}_0)_\beta.
\end{align*}
For fixed $x=x_0$ and variable $\hat{v}\in S_{x_0}^*\p M$, this
expression can be regarded as a non-homogeneous quadratic function
$F(\hat{v})$. We wish to show that at $x=x_0$, the value
$F(\hat{v}_0)$ is uniquely determined.  The fact that condition A is
open would then imply that the values $F(\hat{v})$ are uniquely
determined for all unit-length $\hat{v}$ in an open cone.  But that
implies that the coefficients of $F$ are uniquely determined.  Hence
$\p_ng^{\alpha\beta}(x_0)$ and $\p_n\zeta_\alpha(x_0)$ would be
uniquely determined.  By again appealing to the fact that condition
A is open in $\p SM$, we could conclude that $\p_ng^{\alpha\beta}$
and $\p_n\zeta_\alpha$ are in fact uniquely determined in a
neighborhood of $x_0$ in $\p M$.

To show that $F(\hat{v}_0)$ is uniquely determined, we consider the
right side of \eqref{iterate1} at $x=x_0$ as $m\rightarrow \infty$.
In the second term, $\p_n\zeta_n$ is independent of $m$ and
\begin{align}\label{bounded}
\p_n^2\rho_m(x_0)=\p_n^2\rho(x_0,y_{s_m})\rightarrow
\p_n^2\rho(x_0,y_*).
\end{align}
Since $(\hat{v}_{s_m})_n=\sin({s_m})\rightarrow 0$, we conclude that
the limit of the second term is $0$:
\begin{align}\label{disappear}
2(\hat{v}_{s_m})_n(\p_n\zeta_n-\p_n^2\rho_{s_m})\rightarrow 0.
\end{align}
It only remains to show that the limit of the first term of the
righthand side is uniquely determined at $x_0$. For each $m$,
$\p_n\rho_m(x)$ is known in $\mathcal{H}_m$.  Hence
$\p_n\p_\alpha\rho_m=\p_{\alpha}\p_n\rho_m$ is also known in
$\mathcal{H}_m$ and in particular at $x_0$ and in the limit
$m\rightarrow \infty$.  Thus, $\p_ng^{\alpha\beta}$ and $\p_n
\zeta_{\alpha}$ are uniquely determined for $x\in \p M$ near $x_0$.

Before calculating higher derivatives, we return to equation
\eqref{iterate1}. All coefficients except those in the far right
hand term are known in $x\in \mathcal{H}_m$. Therefore, by solving
the equation, we find that $\p_n\zeta_n-\p_n^2\rho_{s_m}$ is
determined in the same set. $\p_n\zeta_n$ is known because $\zeta_n$
is known in a neighborhood of $\p M$. So we find that
$\p_n^2\rho_{s_m}(x)$ is also determined in $\mathcal{H}_m$.  Hence
all coefficients of \eqref{iterate1} are in fact uniquely determined
in $\mathcal{H}_m$.

To recover higher order normal derivatives of $g$ and $\zeta$, we
continue this process, applying $\p_n$ to equation \eqref{pn} to
obtain
\begin{align}\label{pnpn}
(\p_n^2g^{\alpha\beta})(\hat{v}_{s_m})_\alpha(\hat{v}_{s_m})_\beta+
4(\p_ng^{\alpha\beta})\p_n(\hat{v}_{s_m})_\alpha(\hat{v}_{s_m})_\beta+
\\
2g^{\alpha\beta}(\p_n^2\zeta_\alpha-\p_n^2\p_\alpha\rho_m)(\hat{v}_{s_m})_\beta+
2g^{\alpha\beta}\p_n(\hat{v}_{s_m})_\alpha\p_n(\hat{v}_{s_m})_\beta+
\notag \\
2(\p_n(\hat{v}_{s_m})_n)^2+2(\hat{v}_{s_m})_n(\p_n^2\zeta_n-\p_n^3\rho_m)=0\notag.
\end{align}
The second, fourth, and fifth terms in this sum have already been
shown to be determined in $\mathcal{H}_m$.  Therefore, so is the
combined sum of the other terms:
\begin{align*}
\p_n^2g^{\alpha\beta}(\hat{v}_{s_m})_\alpha(\hat{v}_{s_m})_\beta+
2g^{\alpha\beta}(\p_n^2\zeta_\alpha-\p_n^2\p_\alpha\rho_m)(\hat{v}_{s_m})_\beta
\\ + 2(\hat{v}_{s_m})_n(\p_n^2\zeta_n-\p_n^3\rho_m).
\end{align*}
We have already shown that $\p_n^2\rho_m$ is determined in
$\mathcal{H}_m$. Therefore, by differentiation, so is
$\p_n^2\p_\alpha\rho_m$.  By arguments similar to those of
\eqref{bounded} and \eqref{disappear}, we find that at $x_0$, as
$m\rightarrow \infty$,
\begin{align*}
2(\hat{v}_{s_m})_n(\p_n^2\zeta_n-\p_n^3\rho_m)\rightarrow
0.
\end{align*}
Therefore, we find that the following quantity is determined at
$x_0$:
\begin{align*}
\p_n^2g^{\alpha\beta}(\hat{v}_0)_\alpha(\hat{v}_0)_\beta+
(2g^{\alpha\beta}\p_n^2\zeta_\alpha)(\hat{v}_0)_\beta.
\end{align*}
As in the previous iteration, we regard this expression as a
quadratic function in $v$.  By appealing to the open nature of
condition A, we find that the coefficients must be uniquely
determined.  This is true at $x_0$, but again since condition A is
open, it is also true in a neighborhood of $x_0$.  Hence
$\p_n^2g^{\alpha\beta}$ and $\p_n^2\theta_\alpha$ are uniquely
determined on the boundary near $x_0$.

Returning to equation \eqref{pnpn}, we now find that all the
remaining unknown terms are in fact uniquely determined in
$\mathcal{H}_m$. In particular, since $\p_n^2\zeta_n$ was already
uniquely determined by lemma \ref{potatbo}, $\p_n^3\rho_m$ must be
determined in $\mathcal{H}_m$.

To determine the higher order derivatives, we successively apply
$\p_n$ to equation \eqref{pnpn}.  At each stage, after discarding
all terms with coefficients that are known in $\mathcal{H}_m$, we
obtain
\begin{align*}
\p_n^{k}g^{\alpha\beta}(\hat{v}_{s_m})_\alpha(\hat{v}_{s_m})_\beta+
2g^{\alpha\beta}(\p_n^{k}\zeta_\alpha-\p_n^{k}\p_\alpha\rho_m)(\hat{v}_{s_m})_\beta
\\ + 2(\hat{v}_{s_m})_n(\p_n^{k}\zeta_n-\p_n^{k+1}\rho_m).
\end{align*}
Following \eqref{bounded} and \eqref{disappear}, we find that at
$x_0$ the third term goes to zero as $m\rightarrow \infty$.  From
the last step in the previous iteration, $\p_n^k\rho_m$ is known in
$\mathcal{H}_m$.  Therefore, by differentiation,
$\p_n^k\p_\alpha\rho_m$ is known at $x_0$ in the limit as
$m\rightarrow \infty$.  The remaining two terms are then regarded as
a quadratic function in $\hat{v}$ and the coefficients
$\p_n^kg^{\alpha\beta}$ and $\p_n^k\zeta_\alpha$ are recovered as in
the previous iterations. This allows us to go back and determine the
coefficients of the third term including $\p_n^{k+1}\rho_m$.  Hence
we find that the full jets of $g$ and $\zeta$ are completely
determined at $x_0$ and even in some small neighborhood about $x_0$.
Since $\Omega=d\zeta$, the same is true for the jet of $\Omega$.

The assumption that $w_*\neq v_0$ was used implicitly in
\eqref{bounded}. In the case that $w_*=v_0$, the argument falls
apart because the length of the curve approaches $0$ as $s\rightarrow 0$
and, consequently, $\p_n^2\rho(x_0,y_{s_m})$ is not uniformly bounded
with respect to $m$.  Hence \eqref{disappear} does not necessarily
hold and the quadratic function $F(\hat{v})$ cannot be isolated.

{\bf Case II}: $w_*=v_0$

\begin{lemma}
For $s\geq 0$ sufficiently small, $\rho(x_0,y_s)$ is uniquely
determined by the scattering data.
\end{lemma}
\begin{proof}
Recall, for small $s$, $y_s$ is the first intersection of
$\gamma_{v_s}$ with $\p M$.  Let $J$ be a small interval
$(0,\delta)$.  We shall show that if $\delta$ is sufficiently small,
then $y_s$ is a smooth curve in $\p M$ for $s\in J$. This will hold
for example if $\p M$ is strictly convex with respect to magnetic
geodesics in the direction of $v_0$.  Strict convexity holds when
\begin{align}\label{convex}
II_{x_0}(v_0)-Y(x_0)v_0
\end{align}
is positive, where $II$ denotes the second fundamental form of $\p
M$ with respect to $g$, and $Y$ denotes the $(1,1)$-tensor
determined by the magnetic field.  However, strict convexity is not
necessary. By Taylor's formula, it suffices to show that the first
nonzero derivative of \eqref{convex} in the direction of $v_0$ is
positive. Note that non-trapping and analyticity preclude the
possibility that all derivatives be zero, for analytic continuation
would then imply that the magnetic geodesic $\gamma_{v_0}$ is
trapped in $\p M$.  If the first nonzero derivative of
\eqref{convex} were negative, then the travel time necessary to
reach $y_s$ from $x_0$ would be bounded from below.  Hence it would
be impossible for a sequence $(y_{s},w_{s})$ to approach $(x_0,v_0)$ as
$s\rightarrow 0$ without $\gamma_{v_0}$ being trapped.

We conclude that $y_s$ is a smooth curve in $\p M$ for $s\in J$.
Therefore, $\rho(x_0,y_s)$ is a smooth function on $J$. By equation
\eqref{rhoder},
\begin{align*}
\frac{d}{ds}\rho(x_0,y_s)=\langle
w_s,\p_s{y_s}\rangle_g-\langle\zeta_{y_s},\p_s{y_s}\rangle,
\end{align*}
where $w_s$ is the vector tangent to $\gamma_{v_s}$ at $y_s$. All
terms on the right side are determined by the scattering data,
therefore, by integration, $\rho(x_0,y_s)$ is determined up to a
constant.  By continuity at $s=0$, we find that the constant must be
$0$, since $\lim\limits_{s\rightarrow 0}\rho(x_0,y_s)=0$.
\end{proof}
On $M_i'$ for $i=1,2$, consider the quadratic function
$F_i(x)v=|v|_{g_i}^2-\langle\zeta_{i}(x),v\rangle$.  Omitting the
subscript $i$, in boundary normal coordinates $F_i$ takes the form:
\begin{align*}
F_i(x)v=g_{\alpha\beta}(x)v^\alpha
v^\beta+(v^n)^2-\zeta_\alpha(x)v^\alpha-\zeta_n(x)v^n.
\end{align*}
$(g_1,\zeta_1)$ and $(g_2,\zeta_2)$ have the same jet at $x_0$ if
and only if
\begin{align}\label{equaljets}
\p_n^kF_1(x_0)v=\p_n^kF_2(x_0)v
\end{align}
for all $k>0$ and all $v$ in a small conic neighborhood of $v_0$ in
$S_{x_0}\p M$.

Suppose equation \eqref{equaljets} fails for some $v_1\in S_{x_0}\p
M$ close to $v_0$. Then by considering the Taylor expansion, there
exists a neighborhood of $(x_0,v)$ for which $F_1(x)v>F_2(x)v$ for
all $x$ with $x^n>0$. Let $(\gamma_{v_s})_i$ denote the magnetic
geodesic in $M_i'$ with initial vector $v_s=\cos(s)v_1+\sin(s)\p_n$.
Then we have
\begin{align*}
\rho_1(x_0,y_s)=\int_{(\gamma_{v_s})_1}F_1(x)v>\int_{(\gamma_{v_s})_1}F_2(x)v\geq\\
\int_{(\gamma_{v_s})_2}F_2(x)v=\rho_2(x_0,y_s).
\end{align*}
The second inequality follows from the fact that $(\gamma_{v_s})_2$
is a minimizing curve for the action functional defined by $F_2$.
This contradicts the fact that $\rho_1(x_0,y_s)=\rho_2(x_0,y_s)$.
\end{proof}
\noindent {\em Remark} :  The analytic assumption was only used in
Case II and only to establish a weak convexity condition. Therefore,
this theorem may be applied to nonanalytic systems where that
condition is added to the hypothesis.

\subsection{Recovering the magnetic system in a band about $\p M$}
The mapping $\phi_0:V_1\rightarrow V_2$ given by \eqref{phi0} is a
magnetic equivalence if $(g_1,\Omega_1)$ and $(g_2,\Omega_2)$ have
the same coefficients when expressed in their respective boundary
normal coordinates.  By Theorem \ref{boundary determination} and
condition A on the first magnetic system, the respective
coefficients have the same jets at least at one point in each component
of $\p M$.  By analytic continuation the coefficients must be equal
on all of $\p M\times (-\epsilon,\epsilon)$.  We conclude
$\phi_0^*g_2=g_1$ and $\phi_0^*\Omega_2=\Omega_1$.

\section{Theorem \ref{length recovery}: Recovery of the travel time data}

\subsection{Jacobi fields and conjugate points}

Given a magnetic geodesic $\gamma : [0, T ]\to M$, let $\mathcal{A}$ be the operator on smooth vector fields $Z$ along $\gamma$
defined by
\[\mathcal{A}(Z) = Z''+R(\gamma',Z)\gamma'-Y(Z')-(\nabla_ZY)(\gamma').\]

A vector field $J$ along $\gamma$ is said to be a \emph{magnetic Jacobi field}
if it satisfies the equations
\begin{equation}\label{A=0}\mathcal{A}(J) = 0
\end{equation}
and
\begin{equation}\label{J'gamma=0}
\lan J',\gamma'\ran =  0.
\end{equation}

A magnetic Jacobi field along a magnetic geodesic $\gamma$ is uniquely
determined by specifying $J$ and $J'$ at a point. To see this, consider
the orthonormal basis defined by extending an orthonormal basis $e_1,\dots,e_n$
at $\gamma(0)$ by requiring that
\begin{equation}\label{ON extension}
 e_i'=Y(e_i)
\end{equation}
along $\gamma$. This extension gives an orthonormal basis at each point since
\[ \frac{d}{dt}\lan e_i,e_j \ran=\lan Y(e_i),e_j \ran+ \lan e_i,Y(e_j) \ran =0.\]
Using this basis,
$ J=\sum f_ie_i $
and we can write equation \ref{A=0} as the system
\[f_j'' + \sum_{i=1}^n f_i'y_{ij} + \sum_{i=1}^n f_i a_{ij}=0\]
where $y_{ij}=\lan Y(e_i),e_j\ran$ and
\[ a_{ij}= \lan \nabla_{\gamma'}Y(e_i)+
R(\gamma',e_i)\gamma'-Y(Y(e_i))-(\nabla_{e_i}Y)(\gamma') ,e_j\ran. \]
This is a linear second order system, and therefore it has a unique solution
for each set of initial conditions. Moreover, since the metric and magnetic fields are analytic, so will be the solutions.

It is easy to see from  the definition that $\gamma'$ is always a magnetic Jacobi field. Unlike the case of straight
geodesics, this is the only magnetic Jacobi field parallel to $\gamma'$.
Another difference from the straight geodesic case is that magnetic Jacobi
fields that are perpendicular to $\gamma'$ at $t=0$ do not stay perpendicular
for all $t$. For this reason we will sometimes consider instead the orthogonal
projection $J^\perp =J-f\gamma'$ where $f=\lan J,\gamma'\ran$. The parallel component
$f\gamma'$ is uniquely determined by $J^\perp$ and
$J(0)$ since
\[ f'= \lan J',\gamma'\ran+\lan J,\gamma''\ran= \lan J, Y(\gamma') \ran= \lan J^\perp,
Y(\gamma') \ran. \]

Magnetic Jacobi fields are the variational fields corresponding to variations
through magnetic geodesics. This can be seen by considering the variation
\[ f(t,s)= \gamma_s(t)= exp^\mu_{\tau(s)}(t\theta(s)) \]
where $\tau(s)$ is any curve with $\tau'(0)=J(0)$ and $\theta(s)$ is a vector
field along $\tau$ with $\theta(0)=\gamma'(0)$ and $\theta'(0)=J'(0)$. It is a straight forward computation to check that the variational field $\frac{\partial f}{\partial s}(t,0)$ is a magnetic Jacobi field (see \cite{He}).\\

There is a close relation between magnetic Jacobi fields and conjugate points, analogous to the straight geodesic case. This can be summarized in the following Proposition, proved in \cite{He}.

\begin{prop}
Let $\gamma_\theta:[0,T]\to M$ be the magnetic geodesic with $\gamma(0)=x$. The point $p=\gamma(t_0)$ is conjugate to $x$ along $\gamma$
if and only if there exist a magnetic Jacobi field $J$ along $\gamma$, not
identically zero, with $J(0)=0$ and $J(t_0)$ parallel to $\gamma'$.

Moreover, the multiplicity of $p$ as a conjugate point is equal to the number
of linearly independent such Jacobi fields.\\
\end{prop}

\subsection{Proof of Theorem \ref{length recovery}}

Consider magnetic geodesics $\gamma_1:[0,T_1]\to M_1$ and $\gamma_2:[0,T_2]\to M_2$  that satisfy condition B and such that $\gamma_2(t)=\phi_0(\gamma_1(t))$ for small $t$. Suppose, moreover, that $\gamma_1$ is never tangent to $\partial M_1$. Since both systems have the same scattering data they must exit the manifold at corresponding points and directions, that is $\gamma_2(T_2)=\phi_0(\gamma_1(T_1))$ and $\gamma_2'(T_2)=\phi_{0*}(\gamma_1'(T_1))$. We will call such geodesics correspondent, and, through $\phi_0$, shall implicitly identify $\gamma_1(0)$ to $\gamma_2(0)$ as well as their neighborhoods and tangent spaces.

We want to compare variations of $\gamma_1$ and $\gamma_2$ by defining for a variation
\[ f_1(t,s)=  exp^\mu_{\tau(s)}(t\theta(s)) \]
of $\gamma_1$ a corresponding variation
\[ f_2(t,s)=  exp^\mu_{\phi_0(\tau(s))}(t\tilde\phi_0(\theta(s))) \]
of $\gamma_2$. Nonetheless, $f_1(T_1,s)$ will be in $V$ for small $s$ and that is enough for our purpose. Therefore $f_1(T_1,s)= f_2 (T(s),s)$ where $T(s):(-\epsilon,\epsilon)\to \R$ is differentiable with $T(0)=T_2$. By taking derivatives with respect to $s$ we see that
\[ \frac{\partial f_1}{\partial s}(T_1,0)= \frac{\partial f_2}{\partial s}(T_2,0) + T'(0)\gamma_2'(T_2), \]
so corresponding Jacobi fields agree at their endpoints up to a term tangent to $\gamma_i'$.

We want to consider all Jacobi fields along $\gamma_1$ with $J(T_1)=0$. For this we define the magnetic Jacobi tensor $\mathcal{J}_1$ with the initial conditions $\mathcal{J}_1(T_1)=0$ and $\mathcal{J}_1'(T_1)=Id$. This is a $(1,1)$ tensor in the space of vectors orthogonal to $\gamma_1'$. For any vector $\eta$ at $\gamma_1(T_1)$ perpendicular to $\gamma_1'$, the vector field $J^\perp(t)= \mathcal{J}_1(\eta)(t)$ is the perpendicular part of the Jacobi field with $J(0)=0$ and  $J'(0)=\eta$.

Correspondingly, define the magnetic Jacobi tensor $\mathcal{J}_2$ with $\mathcal{J}_2(T_2)=0$ and $\mathcal{J}_2'(T_2)=Id$. Since correspondent Jacobi fields agree at their endpoints, up to a parallel component, so will the corresponding Jacobi tensors. Therefore $\mathcal{J}_1(t)=\mathcal{J}_2(t)$ for small $t$. Moreover, this is an analytic function and we can use analytic continuation to see that they must agree for all $t$, as long as both are well defined.

Assume, without loss of generality, that $T_1<T_2$. Then, $\mathcal{J}_1(t)=\mathcal{J}_2(t)$ for $0\leq t \leq T_1$. In particular, at $T_1$ we have $\mathcal{J}_2(T_1)=\mathcal{J}_1(T_1)=0$. which would imply that $\gamma_2(T_1)$ is conjugate to  $\gamma_2(T_2)$, of order $n-1$. Since $\gamma_1$ satisfies condition B this can't be the case, therefore $T_2=T_1$.\\

To see that the length of all magnetic geodesics agree, note that magnetic geodesics that are tangent to $\partial M$ form the boundary of $\mathcal{G}$. In particular they are limit points, and form a set of measure $0$. Therefore, we know that the lengths of $\gamma_1$ and $\gamma_2$ agree for a dense subset of $\mathcal{G}_1$. The length of geodesics is a continuous function on $\mathcal{G}$, so the difference in lengths is also a continuous function that is $0$ in a dense subset. This implies that the lengths agree for all magnetic geodesics.

\subsection{About the necessity of condition B}

\begin{figure}
\psfrag{S2}{$M_1$}
\psfrag{RP}{$M_2$}
\psfrag{g}{$\gamma$}
\hspace{4.3cm}\includegraphics[scale=0.7]{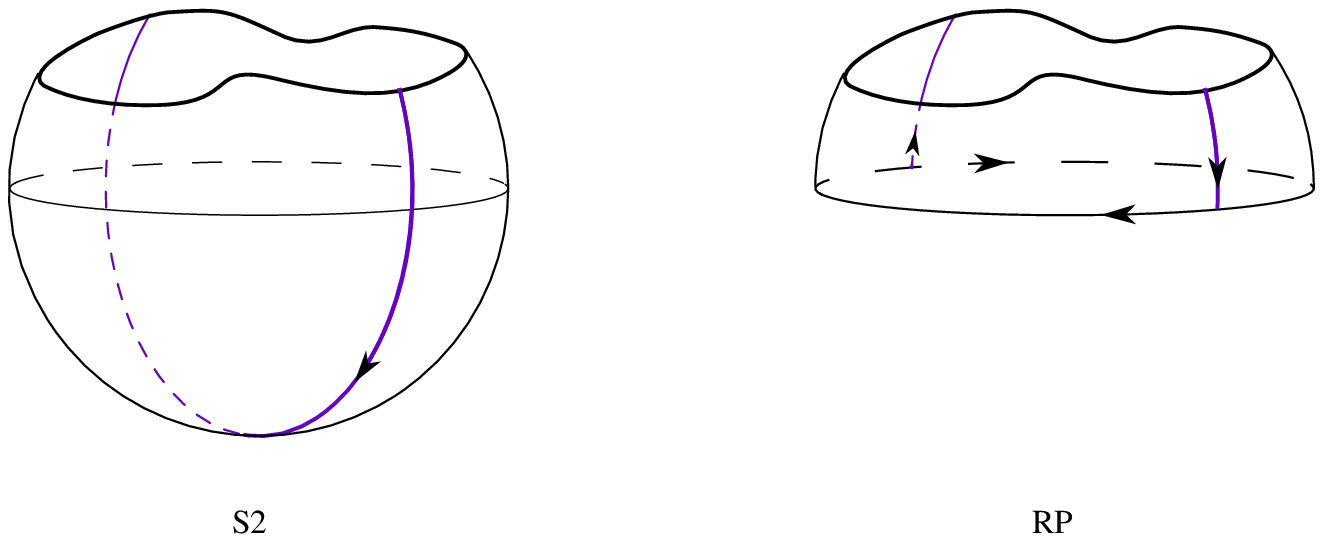}
\caption{} \label{example1Figure}
\end{figure}

In the presence of conjugate points of order $n-1$ the proof above would say that the lengths may differ by the distance between conjugate points of such order. and that $\mathcal{J}_2(T_1-s)=\mathcal{J}_1(T_1-s)=\psi\circ \mathcal{J}_2(T_2-s)$, where $\psi$ is a local isometry of $TM_2$ allowing the possibility that the diffeomorphism $\tilde\phi_0$ does not extend in a compatible way through $\gamma_1$ and $\gamma_2$. So $\mathcal{J}_2$ has a period, up to isometry, of length $T_2-T_1$. \\

In the case of straight geodesics ($\Omega=0$) we can see this behavior by looking at the following examples.

Let $M_1$ be an analytic subset of the sphere $S^2$ that contains the southern hemisphere. Let $M_2$ be a region of $\R P^2$ with the same boundary. We can see it as removing the southern hemisphere from $M_1$ and identifying antipodal points in the equator (See figure \ref{example1Figure}). This manifolds have the same scattering data but any geodesic in $M_1$ that passes through the southern hemisphere will be longer by $\pi$ than the corresponding geodesic in $M_2$. \\

Another similar example is when $N_1$ is a subset of the sphere $S^2$ that contains both poles and a meridian $\tau$. To build $N_2$ we cut open along $\tau$ and glue another $S^2$ opened at the meridian in such a way that whenever you cross $\tau$ you move from one sphere to the other (see figure \ref{example2Figure}). These manifolds have the same scattering data but any geodesic in $N_1$ that passes through $\tau$ will be shorter by $2\pi$ than the corresponding geodesic in $N_2$. This is not a Riemannian manifold, since the poles are singular points. But it is orientable and the Jacobi tensors will be periodic with a period of $2\pi$. \\

The need of condition B is clear in the examples above, where we can see that
the presence of a conjugate point of order $n-1$ allows the length of a geodesic
to change. The question remains if it is necessary as an independent hypothesis.
It seems to be that, like in the examples above, conjugate points of order $n-1$
only appear when there are trapped geodesics. If this is true, it would make condition
B empty in this setting. \\

\begin{figure}
\psfrag{N}{}
\psfrag{S}{}
\psfrag{T}{$\tau$}
\psfrag{g}{$\gamma$}
\psfrag{N1}{$N_1$}
\psfrag{N2}{$N_2$}
\hspace{4.3cm}\includegraphics[scale=0.65]{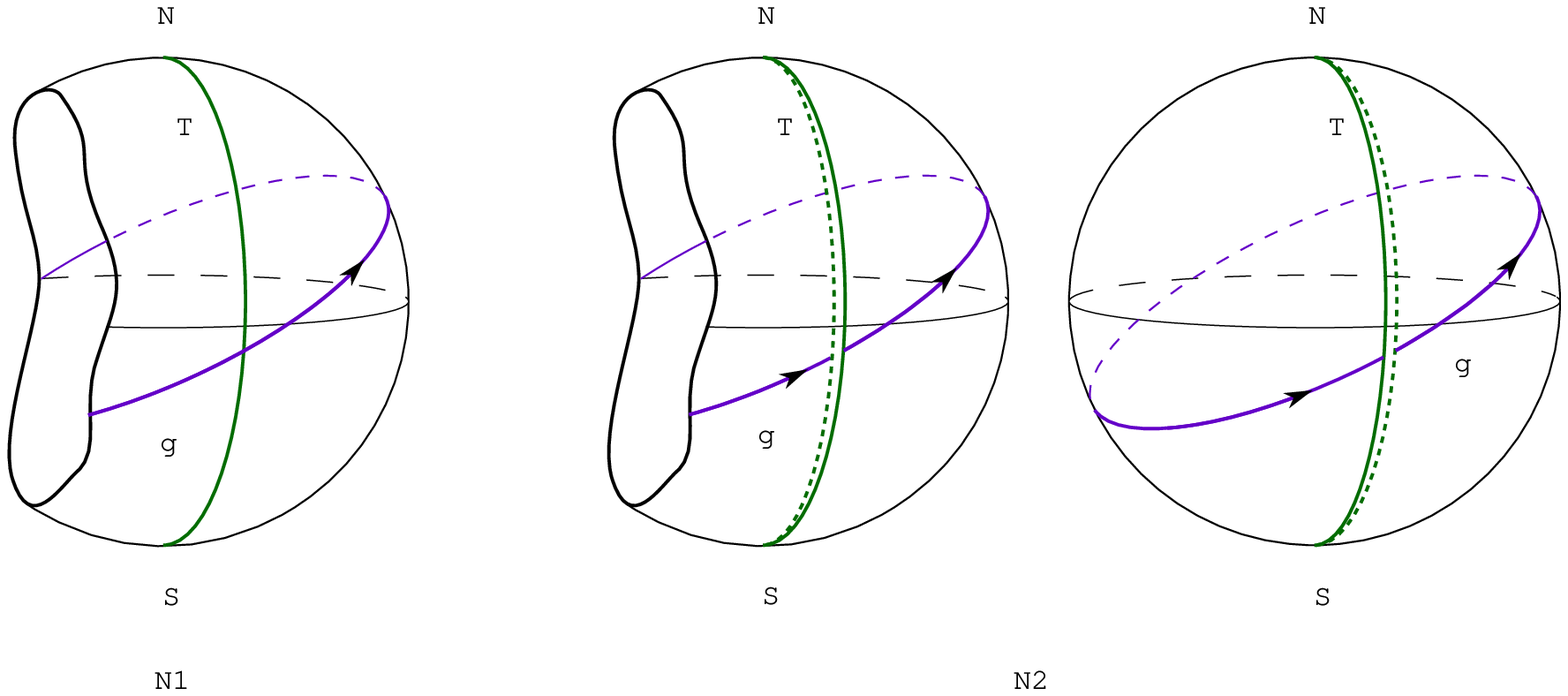}
\caption{} \label{example2Figure}
\end{figure}

\section{Theorem \ref{equivalence construction}: Construction of the magnetic equivalence}
We have two non-trapping, analytic magnetic systems
$(M_i,g_i,\Omega_i)$ with the same scattering data and travel time
data.  Moreover, the mapping $\phi_0:V_1\rightarrow V_2$ is a
magnetic equivalence.  To prove Theorem \ref{equivalence
construction}, we must extend $\phi_0$ to a magnetic equivalence
$\phi:M_1'\rightarrow M_2'$.

This has been done in the absence of magnetic fields in \cite{V}.
The proof carries over in the presence of magnetic fields with only
changes in notation. Therefore, we shall only sketch the proof here.

\subsection{Extension of $\phi_0$ to $M_1'$}
\begin{proof}
First we define a mapping $\ti{\phi}:SM_1\rightarrow M_2$.  We then
show that all vectors lying over the same base point in $M_1$ map to
the same point in $M_2$, and that the induced map
$\phi:M_1\rightarrow M_2$ agrees with $\phi_0$ on their common
domain.

Let $(x,\theta)\in SM_1$, and consider the magnetic geodesic
$\gamma_{x,\theta}(t)$.  Going backwards along $\gamma_{x,\theta}$,
we will eventually leave $M_1$ by the non-trapping assumption. We
define
\begin{align*}
T_0=T_0(x,\theta)=\inf\{t\geq 0:\gamma(x,\theta)(-t)\notin M_1\}.
\end{align*}
$T_0\geq 0$ satisfies $\gamma_{x,\theta}(-T_0)\in \p M$. Moreover,
by the fact that $\p M$ is analytic, it is not hard to see that
$\gamma_{x,\theta}(-T_0-s)\in V_1\setminus M_1$ for all $s$ such
that $0<s<\delta$ for some sufficiently small
$\delta=\delta(x,\theta)$. We arbitrarily choose
\begin{align*}
-T=-T(x,\theta)\in (-T_0-\delta,-T_0),
\end{align*}
and let $(z_{x,\theta},\xi_{x,\theta})\in SM_1'$ be given
by
\begin{align*}
(z_{x,\theta},\xi_{x,\theta})=(\gamma_{x,\theta}(-T),\gamma_{x,\theta}'(-T)).
\end{align*}
Let $(y_{x,\theta},\eta_{x,\theta})\in SM_2'$ be its image by
$\phi_0$,
\begin{align*}
(y_{x,\theta},\eta_{x,\theta})=\phi_{0*}(z_{x,\theta},\xi_{x,\theta}).
\end{align*}
Define
\begin{align*}
\ti{\phi}(x,\theta)=\gamma_{y_{x,\theta},\eta_{x,\theta}}(T).
\end{align*}

\begin{proposition}
$\ti{\phi}$ is a well defined function on $SM_1$ with values in
$M_2$.
\end{proposition}
\begin{proof}
It must be shown that the value of $\ti{\phi}(x,\theta)$ is
independent of the choice of $T\in (-T_0-\delta,-T_0)$, and that the
values do belong to $M_2$. The first statement follows from the fact
that the two magnetic systems are equivalent via $\phi_0$ in $V$.
The second follows from the fact that the two systems have share the
same length data. See \cite{V} for details.
\end{proof}

\begin{proposition}
The value of $\ti{\phi}(x,\theta)$ is independent of the direction
$\theta$.
\end{proposition}

\begin{lemma}\label{dsquared}
Let $M'$ be an open manifold with analytic metric $g$.  Then for
every $x_0\in M'$, there exists a positive number $r$ such that the
squared distance function is analytic on the set
\[\Delta_r(x_0)=\{(x,y):d(x,x_0)<r, d(x,y)<r\}.\]
If $K$ is a compact set contained within the interior of $M'$, then
there is an open $O\subset M'$ containing $K$ and a positive number
$r$ such that the squared distance function is analytic on the set
\begin{align*}
\Delta_{O,r}(K)=\{(x,y):x\in O, d(x,y)<r\}.
\end{align*}
\end{lemma}
See \cite{V} for proof.

\begin{proof}
Since the sphere $|\theta|_g=1$ is connected, it is sufficient to
show that for fixed $x$, $\ti{\phi}(x,\theta)$ is locally constant
in $\theta$.

Fix $(x_0,\theta_0)$, and let $N\subset S_{x_0}M_1$ be a
neighborhood of $\theta_0$.  Let $(z_0,\xi_0)$ and $(y_0,\eta_0)$
correspond to $(x_0,\theta_0)$ following the notation above.  Then
we have
\begin{align*}
x_0=&\gamma_{z_0,\xi_0}(T) \\
\ti{\phi}(x_0,\theta_0)=&\gamma_{y_0,\eta_0}(T).
\end{align*}
For all $\theta\in N$, let
$(z,\xi)=(\gamma_{x_0,\theta}(-T),\gamma_{x_0,\theta}'(-T))$, so
that
\begin{align*}
x_0=\gamma_{z,\xi}(T).
\end{align*}
Let $(y,\eta)=\phi_{0*}(z,\xi)$. It can be shown that if $N$ is
sufficiently small,
\begin{align}\label{nearby}
\ti{\phi}(x_0,\theta)=\gamma_{y,\eta}(T).
\end{align}
Equation \eqref{nearby} seems to be a direct application of the
definition of $\ti{\phi}$ but in fact, between $x_0$ and $z$, the
curve $\gamma_{x_0,\theta}$ can pass across the boundary in and out
of $M_1$.  Therefore, $T$ is not necessarily a valid choice for
$T(x_0,\theta)$ according to the definition. Nevertheless, using the
fact that the two magnetic systems have the same scattering and
length data, equation \eqref{nearby} still holds.

According to lemma \ref{dsquared}, the functions
\begin{align*}
\rho_1(t)=&d_{g_1}^2(\gamma_{z_0,\xi_0}(t), \gamma_{z,\xi}(t)),
\quad \text{and} \\
\rho_2(t)=&d_{g_2}^2(\gamma_{y_0,\eta_0}(t),\gamma_{y,\eta}(t))
\end{align*}
will be analytic wherever their values are less than some fixed
constant $r$.

Since $\phi_0$ is a magnetic equivalence, it carries magnetic
geodesics into magnetic geodesics.  $(y,\eta)=\phi_{0*}(z,\xi)$.
Therefore, for $t$ near $0$, we have
\begin{align*}
\phi_0(\gamma_{z,\xi}(t))=\gamma_{y,\eta}(t).
\end{align*}
 A magnetic equivalence is, in particular, an
isometry of metrics, so we conclude that for $t$ near $0$,
$\rho_1(t)=\rho_2(t)$.  By analytic continuation, the two function
must be equal for all $t$ up to $T$. Therefore,
$\rho_2(T)=\rho_1(T)=0$. This implies that
\begin{align*}
\ti{\phi}(x_0,\theta)=\gamma_{y,\eta}(T)=\gamma_{y_0,\eta_0}(T)=\ti{\phi}(x_0,\theta_0).
\end{align*}
\end{proof}
We define $\phi:M_1'\rightarrow M_2'$ by
\begin{align*}
x\mapsto & \ti{\phi}(x,\theta),\quad &\text{for}\quad x\in M_1 \\
x\mapsto & \phi_0(x),\quad &\text{for}\quad x\in M_1'\setminus M_1.
\end{align*}
The magnetic geodesic flows on the respective sphere bundles of
$M_1$ and $M_2$ are analytic.  This implies that $\phi$ is an
analytic mapping. The fact that $\phi|_{V_1}=\phi_0$ follows
directly from the definition of $\phi$ and the fact that $\phi_0$ is
a magnetic equivalence.  By analytic continuation, $\phi$ must
satisfy
\begin{align*}
\phi^*g_2=g_1,\quad \phi^*\Omega_2=\Omega_1.
\end{align*}
\end{proof}

\end{document}